\documentclass{tac}

\usepackage[mathletters]{ucs}
\usepackage[utf8]{inputenc}

\usepackage{graphicx}

\usepackage{natbib}
\usepackage{graphicx}
\usepackage{amsmath}
\usepackage{amssymb}
\usepackage{tikz-cd}
\usepackage{xcolor}
\usepackage{quiver}
\usepackage{afterpage}

\usepackage{listings}

\lstdefinestyle{mystyle}{
  showspaces=false,                
  showstringspaces=false,
  showtabs=false,                  
  tabsize=2,
  basicstyle=\ttfamily\footnotesize,
  breakatwhitespace=false
  }
\newcommand{\tensor}{\otimes}
\newcommand{\iso}{\cong}
\newcommand{\lscat}[1]{\mathscr{#1}}
\newcommand{\scat}[1]{\mathcal{#1}}
\newcommand{\E}{\mathbb{E}}
\newcommand{\F}{\mathbb{F}}
\renewcommand{\L}{\scat{L}}
\newcommand{\T}{\mathbb{T}}

\newcommand{\One}{\lscatfont{1}}

\newcommand{\Nat}{\mathbb{N}}
\newcommand{\lscatfont}[1]{{\boldsymbol{#1}}}
\newcommand{\Set}{{\lscatfont{\mathcal{S}\hspace{-.5mm}\mathit{et}}}}
\newcommand{\FF}{\lscatfont{\mathcal{F}}}
\newcommand{\id}{\mathrm{id}}
\newcommand{\comp}{\circ}
\newcommand{\icomp}{\,}
\newcommand{\ord}[1]{[#1]}
\newcommand{\setof}[1]{\{ #1 \}}
\newcommand{\suchthat}{\mid}
\newcommand{\op}{\mathrm{op}}
\newcommand{\Def}[1]{\emph{#1}}
\newcommand{\Id}{\mathrm{Id}}
\newcommand{\AbsClon}{\lscatfont{\mathcal{AC}}}
\newcommand{\SubstAlg}{\lscatfont{\mathcal{S\!A}}}
\newcommand{\Var}{\mathrm{V}}
\newcommand{\old}{\mathrm{old}}
\newcommand{\new}{\mathrm{new}}
\newcommand{\weak}{\mathrm{w}}
\newcommand{\symm}{\mathrm{s}}
\newcommand{\swap}{\symm}
\newcommand{\cont}{\mathrm{c}}
\newcommand{\monswap}{\underline{\swap}}
\newcommand{\monweak}{\underline{\weak}}
\newcommand{\moncont}{\underline{\cont}}
\newcommand{\Endo}{\mathrm{Endo}}
\newcommand{\str}{\mathrm{str}}

\newcommand{\ACtoSA}{\mathrm{S}}
\newcommand{\SAtoAC}{\mathrm{C}}

\newcommand{\pt}{^\bullet}

\author{Marcelo Fiore and Sanjiv Ranchod}

\thanks{Marcelo Fiore was partially supported by EPSRC grant EP/V002309/1.
  Sanjiv Ranchod was supported by the Gates Cambridge Trust.}

\address{Department of Computer Science and Technology, University of
  Cambridge, William Gates Building, 15 JJ Thomson Avenue, Cambridge CB3 0FD,
  UK.}

\title{A Finite Algebraic Presentation of Lawvere Theories\\ in the
  Object-Classifier Topos}

\copyrightyear{2024}


\eaddress{marcelo.fiore@cl.cam.ac.uk \CR sr2008@cl.cam.ac.uk}

\begin{document}

\maketitle

\begin{center}
  \emph{Dedicated to Bill Lawvere with gratitude and admiration}
\end{center}

\bigskip

\begin{abstract}
Over the topos of sets, the notion of Lawvere theory is infinite
countably-sorted algebraic but not one-sorted algebraic.  
Shifting viewpoint over the object-classifier topos, a finite algebraic
presentation of Lawvere theories is considered.
\end{abstract}

\section{Introduction}
\label{Section:Introduction}

The notion of Lawvere theory was introduced by Bill Lawvere in his seminal
PhD~thesis~\cite[Chapter~II, Section~1]{LawverePhDThesis}.  It initiated the
subject of categorical algebra while transforming the subject of universal
algebra~(\cite{Wraith,AdamekRosickyVitale}).

A pivotal role in the definition of Lawvere theory is played by the category
$\E = \F^\op$, for $\F$ the category with objects the set of natural numbers
$\Nat$ and morphisms $m\to n$ in $\F$ given by functions $\ord m \to \ord n$
where $\ord \ell = \setof{\, k\in\Nat \suchthat k<\ell \,}$.

\begin{definition}
A \Def{Lawvere theory} is a pair $(\L,L)$ with $L: \E \to \L$ an
identity-on-objects finite-product-preserving functor.
\end{definition}

Straightaway, Bill Lawvere introduced and studied the category 
of Lawvere theories.  A Lawvere-theory morphism is a functor between co-slices
under $\E$.
The most basic example of a Lawvere theory is the initial one $(\E,\Id_\E)$.
The terminal Lawvere theory $(\T,T)$ arises from the identity-on-objects and
fully-faithful factorization $\E\to\T\to\One$ of the unique functor from $\E$
to the terminal category $\One$.

The connection with universal algebra was established by Bill Lawvere from the
outset.  For instance, he showed, both in general and in examples, how Lawvere
theories may be described by means of equational
presentations~\cite[Chapter~II, Section~2]{LawverePhDThesis}.  
%
%
In this paper, 
we will use the notation $x_1,\ldots,x_n\vdash t = u$ for the equation that
identifies the terms $t$ and $u$, both with free variables amongst
$x_1,\ldots,x_n$.

The initial Lawvere theory, being presented by no operations and no equations,
was defined by Bill Lawvere as the \Def{theory of equality}.  
On the other hand, the terminal Lawvere theory may be presented by means of
any non-empty set of constants $O$ subject to the equations $x\vdash o=x$ for
each constant $o$ in $O$.  This example displays an important contrast, often
emphasized by Bill Lawvere, between Lawvere theories and equational
presentations: the former are representation independent.  Indeed, there are
in general many equational presentations for the same Lawvere theory and each
Lawvere theory determines one variety of algebraic structure or algebraic
category~\cite[Chapter~III]{LawverePhDThesis}.

Bill Lawvere defined a Lawvere theory $(\L,L)$ to be inconsistent whenever
$\L$ is equivalent to either the terminal category 
or the arrow category. 
The terminal theory $(\T,T)$ is an example of the first kind.  An example of
the second kind is the theory 
presented by no operators subject to the equation $x,y\vdash x = y$.  This is
a non-initial Lawvere sub-theory of the terminal Lawvere theory.  Therefore,
the category of Lawvere theories has a non-initial sub-terminal object. The
fact below then follows from~\cite[Lemma~11.19]{AdamekRosickyVitale}. 

\begin{lemma}
The category of Lawvere theories 
is not one-sorted algebraic. 
\end{lemma}
\noindent
In other words, the language of (one-sorted) universal algebra is not
expressive enough to describe itself.

It is well-known, however, that the category of Lawvere theories is
countably-sorted algebraic; see, for
instance,~\cite[Remark~14.25(1)]{AdamekRosickyVitale}.  
One way to see this in direct connection to universal algebra is by means of
the notion of \emph{abstract clone}; for which see, for
instance,~\cite{Cohn,Taylor93, Gratzer}.  
In the context of many-sorted universal algebra, we will use the notation
${x_1:s_1},\ldots,{x_n:s_n} \vdash t = u:s$ for the equation that identifies
the terms $t$ and $u$ of sort $s$, both with free variables amongst
$x_1,\ldots,x_n$ respectively of sort $s_1,\ldots,s_n$.

The countably-sorted equational presentation of abstract clones has set of
sorts $\Nat$ and operators
\[
\mu_{m, n} : m , \underbrace{n,\ldots,n}_{\text{$m$ times}} \to n
  \enspace(m, n \in \Nat)
\enspace,\quad
\iota^m_i : m \enspace\big(m \in \Nat, i\in\ord m\big)
\]
subject to the equations
\[\begin{array}{l}
\begin{array}{l}
x:\ell, y_0:m, \ldots, y_{\ell-1}:m, z_0:n , \ldots, z_{m-1}:n
\\[1mm]
\quad \vdash \ 
\mu_{m,n}( \mu_{\ell,m}(x, y_0, \ldots, y_{\ell-1} ), z_0, \ldots, z_{m-1} ) 
\\[1mm]
\quad\qquad = \
\mu_{\ell,n}
  ( x , 
    \mu_{m, n}( y_0, z_0, \ldots, z_{m-1} ), 
    \ldots, 
    \mu_{m, n}( y_{\ell-1}, z_0, \ldots, z_{m-1} )
    )
: n
\end{array}
\\[9mm]
\begin{array}{l}
x_0:n, \ldots, x_{m-1}:n
\vdash \mu_{m, n}( \iota^{m}_i, x_0, \ldots, x_{m-1} ) =  x_i : n
\end{array}
\\[3mm]
\begin{array}{l}
x:m 
\vdash \mu_{m, m}( x, \iota^{m}_0, \ldots, \iota^{m}_{m-1} ) =  x  : m
\end{array}
\end{array}\]
The idea behind this axiomatization is that the operators $\mu_{m,n}$ model
simultaneous substitution (or cartesian multi-composition) while the constants
$\iota^n_i$ model variables (or projections).

The categories of Lawvere theories 
and of abstract clones 
are equivalent\footnote{This was in fact known to Bill Lawvere from the
  outset.  Indeed, he once mentioned to the first author at a category theory
  conference that he had actually come up with the notion of abstract clone in
  the process of developing the notion of Lawvere theory.}; see, for
instance,~\cite[Appendix]{Taylor73}.
Concisely put, the abstract clone of a Lawvere theory $\L$ consists of the
sorted family of sets $\setof{\, \L(n,1) \,}_{n\in\Nat}$ equipped with the
operations
\[
  \L(m,1) \times \L(n,1)^m 
  \stackrel\iso\longrightarrow
  \L(m,1) \times \L(n,m) 
  \stackrel\comp\longrightarrow
  \L(n,1)
  \quad
  (m, n \in \Nat)
\]
\[
  \pi_{i+1}
    \in \L(m,1)
  \quad
  \big(m \in \Nat, i\in\ord m\big)
\]
while the Lawvere theory of an abstract clone 
\[
  \setof{\, C_n \,}_{n\in\Nat}
  \enspace,\quad 
  \mu_{m,n} : C_m \times (C_n)^m \to C_n \enspace (m,n \in\Nat)
  \enspace,\quad 
  \iota^m_i \in C_m \enspace \big(m \in \Nat, i\in\ord m\big)
\]
has hom-sets $(C_m)^n$, from $m$ to $n$ in $\Nat$, with composition
\[
  (C_m)^n \times (C_\ell)^m
  \xrightarrow{\ \langle \pi_k \times \id \rangle_{k=1,\dots,n}\ }
  \big( C_m \times (C_\ell)^m \big)^n
  \xrightarrow{\ (\mu_{m,\ell})^n\ }
  (C_\ell)^n
  \quad
  (\ell,m,n\in\Nat)
\]
and identities
\[
  (\iota^m_0,\ldots,\iota^m_{m-1}) \in (C_m)^m
  \quad
  (m\in\Nat)
\]

\medskip
We have so far confined our discussion of categorical algebra to universal
algebra; that is, on sets.  Of course, one of the benefits of the categorical
approach is the generalization to other realms.  This was recognized early on
by Bill Lawvere.  In particular, in~\cite{LawvereOrdSum}, he put forward the
study of more general equational structure emphasizing that this necessitates
operations with both arities and co-arities, where the latter embody
generalized tupling, parameterization, or indexing.  
In this paper, we consider such algebraic structure in the object-classifier
topos $\FF=\Set^\F$.  Specifically, we present a finite equational
presentation of Lawvere theories over $\FF$.  This result is implicit
in~\cite{FPT} Theorem~3.3 and Proposition~3.4, and we use this occassion to
dedicate it to Bill Lawvere.  
The corresponding \emph{theory of substitution} has two operators,
respectively with arity-coarity pairs $(\,V+1\,,\,1\,)$ and $(\,0\,,\,V\,)$,
where $V = \F(1,-)$ is the universal object model, subject to four equations.
A substitution algebra is then an object $A$ in $\FF$ together with operations
$s: A^V \times A \to A$ and $v: 1 \to A^V$ satisfying natural laws when $s$ is
understood as modelling (single variable) substitution and $v$ as (generic)
variables.  The use of the non-standard arity $V+1$ and of the non-standard
co-arity $V$ is fundamental.  

The present development streamlines that of~\cite{FPT} from the viewpoint of a
universal characterization of 
$\F$ as a monoidal theory that emphasizes its structural properties
(weakening, contraction, and exchange) and serves as the conceptual
foundation~(c.f.~\cite{FOSSACS05, MFPS06}).  
In this vein, Sections~\ref{Section:SymmetricMonoidsAndMonads}
and~\ref{Section:SymmetricDistributiveLaws} introduce the necessary theory of
symmetric (co)monads and distributive laws involving them.  
Within this framework, Section~\ref{Section:SubstitutionAlgebras} recasts and
generalizes the notion of substitution algebra. 
Finally, Section~\ref{Section:IsomorphismTheorem} outlines, for the first time
in print, the isomorphism between abstract clones and substitution algebras.
This establishes the main aim of the paper in providing a finite equational
presentation of Lawvere theories in the object-classifier topos.  From the
viewpoint of many-sorted universal algebra, it may be reinterpreted as
providing a countably-sorted algebraic presentation of Lawvere theories by an
axiomatization of single-variable substitution.

\medskip
Along the above lines of enquiry, we leave open the conjecture that the notion
of Lawvere theory is truly infinite countably-sorted algebraic, 
in that there are no countably-sorted equational presentations of Lawvere
theories with either a finite set of sorts, a finite set of operators, or a
finite set of equations.

\section{Symmetric monoids and monads}
\label{Section:SymmetricMonoidsAndMonads}

\subsection{} \label{Subsection:MainSymmetricMonoid}
In the spirit of~\cite{LawverePhDThesis}, in~\cite{FPT}, the category $\F$ is
viewed as the free cocartesian category on an object $1$ with a chosen
(strict) coproduct structure 
\[\begin{tikzcd}[ampersand replacement=\&]
	{n} \& {n + 1} \& {1}
	\arrow["{\old_n}", from=1-1, to=1-2]
	\arrow["{\new_n}"', from=1-3, to=1-2]
\end{tikzcd}\] 
Via this coproduct structure, every morphism in $\F$ may be described using
the following generating morphisms:
\begin{align*}
\cont = \; & [\id_1 , \id_1 ] : 2 \to 1
\\
\weak = \; & \old_0 : 0 \to 1
\\
\symm = \; & [ \new_1 , \old_1 ] : 2 \to 2
\end{align*}
As is well-known, the first two maps equip $\F$ with the monad 
$( {\Id+1}, \id+\weak, \id+\cont )$ taking a coproduct with $1$.  This
description, however, overlooks the map $\symm$.  To account for 
it, 
we consider the work of~\cite{Grandis} where $\F$ is instead viewed as the
free strict monoidal category on a chosen ``symmetric monoid''.

\begin{definition}
Let $(\lscat{C}, \tensor, I)$ be a monoidal category. A \Def{symmetric monoid}
$(A, c, w, s)$ in $\lscat{C}$ consists of an object $A$ of $\lscat{C}$ and
morphisms $c: A\otimes A \to A$, $w: I \to A$, and 
$s: A\otimes A \to A \otimes A$ satisfying the following commutative
diagrams (where associators and unitors have been omitted):
\[\begin{tikzcd}[ampersand replacement=\&]
  {A^{\tensor 3}} \& {A^{\tensor 2}} \& A \& {A^{\tensor 2}} \& 
  {A^{\tensor 2}} \& {A^{\tensor 2}} \& {A^{\tensor 2}} \& {A^{\tensor 2}} \\ 
  {A^{\tensor 2}} \& A \& {A^{\tensor 2}} \& A \&\& A \&\&
  {A^{\tensor 2}}
	\arrow["{c\tensor \id}", from=1-1, to=1-2]
	\arrow["{\id\tensor c}"', from=1-1, to=2-1]
	\arrow["c"', from=2-1, to=2-2]
	\arrow["c", from=1-2, to=2-2]
	\arrow["{\id\tensor w}", from=1-3, to=1-4]
	\arrow["{w\tensor \id}"', from=1-3, to=2-3]
	\arrow["c", from=1-4, to=2-4]
	\arrow["c"', from=2-3, to=2-4]
	\arrow["s", from=1-5, to=1-6]
	\arrow["c"', from=1-5, to=2-6]
	\arrow["c", from=1-6, to=2-6]
	\arrow["s", from=1-7, to=1-8]
	\arrow["\id"', from=1-7, to=2-8]
	\arrow["s", from=1-8, to=2-8]
	\arrow["\id", from=1-3, to=2-4]
\end{tikzcd}\]
\[\begin{tikzcd}[ampersand replacement=\&]
	{A^{\tensor 3}} \& {A^{\tensor 3}} \& {A^{\tensor 3}} \& A \& 
    {A^{\tensor 2}} \& {A^{\tensor 3}} \& {A^{\tensor 3}} \& {A^{\tensor 3}} \\
	{A^{\tensor 3}} \& {A^{\tensor 3}} \& {A^{\tensor 3}} \&\& {A^{\tensor 2}} \&
  {A^{\tensor 2}} \&\& {A^{\tensor 2}}
	\arrow["{s\tensor \id}", from=1-1, to=1-2]
	\arrow["{\id\tensor s}", from=1-2, to=1-3]
	\arrow["{\id\tensor s}"', from=1-1, to=2-1]
	\arrow["{s\tensor \id}"', from=2-1, to=2-2]
	\arrow["{\id\tensor s}"', from=2-2, to=2-3]
	\arrow["{s\tensor \id}", from=1-3, to=2-3]
	\arrow["{w\tensor \id}", from=1-4, to=1-5]
	\arrow["{\id\tensor w}"', from=1-4, to=2-5]
	\arrow["s", from=1-5, to=2-5]
	\arrow["{s\tensor \id}", from=1-6, to=1-7]
	\arrow["{ \id\tensor s}", from=1-7, to=1-8]
	\arrow["{ \id\tensor c}"', from=1-6, to=2-6]
	\arrow["s"', from=2-6, to=2-8]
	\arrow["{c\tensor \id}", from=1-8, to=2-8]
\end{tikzcd}\]
\end{definition}
Note, interestingly, that this definition does not require the tensor product
to be braided and, when it is, it may be specialized to the expected
definition of a commutative monoid.  

\subsection{Symmetric monads, comonoids, and comonads}
Recalling that the category of endofunctors $\Endo(\lscat{C})$ on a category
$\lscat{C}$ is strict monoidal,
the above allows for the definition of a \Def{symmetric monad} on $\lscat{C}$
as a symmetric monoid in $\Endo(\lscat{C})$.  There are, of course,
appropriate dual definitions of a \emph{symmetric comonoid} and
\emph{symmetric comonad}. 

\begin{examples} \label{Example:SymmetricMonoidMonad}
\begin{enumerate}
\item 
For a symmetric monoid $A$, the tensoring with $A$ functor $(-)\otimes A$ is a
symmetric monad.
In particular, as every object $A$ in a cartesian monoidal category is
canonically a symmetric comonoid, the product with $A$ comonad $(-)\times A$
is canonically symmetric.

\item 
For a symmetric comonoid $A$ in a monoidal closed category, the internal-hom
functor $[A,-]$ is a symmetric monad.
In particular, for every object $A$ in a cartesian monoidal closed category,
the exponentiation by $A$ monad $(-)^A$ is canonically symmetric.

\item 
Monoidal functors preserve symmetric (co)monoids.  
In particular, by 
the 
convolution monoidal structure~(\cite{Day,ImKelly}), representable presheaves
of symmetric (co)monoids are symmetric (co)monoids.

\item \label{Example:ConcreteSymmetricMonoidMonad}
A concrete class of examples arises from the general ones above as follows.

Let $(A,c,w,s)$ be a symmetric monoid in a monoidal small category $\lscat C$.
Then, the representable $R_A = \lscat C(A,-)$ in $\Set^{\lscat C}$ is a
symmetric comonoid for the convolution monoidal structure and the convolution
internal-hom functor $[R_A,-]$ on $\Set^{\lscat C}$ is a symmetric monad.
In fact, it is canonically isomorphic to the symmetric monad 
$(-\tensor A)^\star = \Set^{(-)\tensor A}$ on $\Set^{\lscat C}$ with the
simple description below:
\begin{align*}
\big(\!-\tensor A\big)^\star(X) =\ 
& X(-\tensor A) 
\\[1mm]
\underline c_X = \ 
& X(- \tensor c) : X(-\tensor A\tensor A) \longrightarrow X(-\tensor A)
\\[1mm]
\underline w_X = \ & X(- \tensor w) : X(-) \longrightarrow  X(-\tensor A)
\\[1mm]
\underline s_X = \ 
& X(- \tensor s) 
  : X(-\tensor A\tensor A) \longrightarrow X(-\tensor A\tensor A)
\end{align*}

\end{enumerate}
\end{examples}

We will need to consider monoidal notions of symmetric monads.

\begin{definition}
Let $(\lscat{C}, \tensor, I)$ be a monoidal category. A \emph{lax monoidal}
symmetric monad on $\lscat{C}$ is a symmetric monad 
$(T, \mu, \eta, \varsigma)$ on $\lscat C$ equipped with a natural
transformation ${\ell_{A, B}: T(A) \otimes T(B) \to T(A \otimes B)}$ and a
morphism ${e : I \to T(I)}$ satisfying, for every ${A, B \in \lscat{C}}$, the
following commutative diagrams
\[\begin{array}{c}
  \begin{array}{c}\begin{tikzcd}[column sep = large, ampersand replacement=\&]
	{T^2(A) \otimes T^2(B)} \& {T(T(A)\otimes T(B))} \& {T^2(A \otimes B)} 
  \\
	{T(A) \otimes T(B)} \&\& {T(A \otimes B)} 
  \\
	\arrow["{\mu_A \otimes \mu_B}"', from=1-1, to=2-1]
	\arrow["{\ell_{T(A), T(B)}}", from=1-1, to=1-2]
	\arrow["{T(\ell_{A, B})}", from=1-2, to=1-3]
	\arrow["{\mu_{A \otimes B}}", from=1-3, to=2-3]
	\arrow["{\ell_{A,B}}"', from=2-1, to=2-3]
\end{tikzcd}\end{array}
\\[-6mm]
\begin{array}{c}\begin{tikzcd}[ampersand replacement=\&]
  {A\otimes B} \\
  {T(A)\otimes T(B)} \& {T(A\otimes B)} \\
	\arrow["{\eta_{A\otimes B}}", from=1-1, to=2-2]
	\arrow["{\eta_A\otimes \eta_B}"', from=1-1, to=2-1]
	\arrow["{\ell_{A,B}}"', from=2-1, to=2-2]
\end{tikzcd}\end{array}
\\[-5mm]
\begin{array}{c}\begin{tikzcd}[column sep = large, ampersand replacement=\&]
	{T^2(A) \otimes T^2(B)} \& {T(T(A)\otimes T(B))} \& {T^2(A \otimes B)} \\
	{T^2(A) \otimes T^2(B)} \& {T(T(A)\otimes T(B))} \& {T^2(A \otimes B)}
	\arrow["{\ell_{T(A), T(B)}}", from=1-1, to=1-2]
	\arrow["{T(\ell_{A, B})}", from=1-2, to=1-3]
	\arrow["{\varsigma_A\otimes \varsigma_B}"', from=1-1, to=2-1]
	\arrow["{\ell_{T(A), T(B)}}"', from=2-1, to=2-2]
	\arrow["{T(\ell_{A, B})}"', from=2-2, to=2-3]
	\arrow["{\varsigma_{A \otimes B}}", from=1-3, to=2-3]
\end{tikzcd}\end{array}
\end{array}\]
as well as expected coherence diagrams involving the unitors and associator of
the monoidal structure.  
\end{definition}

A symmetric monad is called \emph{oplax monoidal} if it satisfies the dual
definition, while it is called \emph{monoidal} in the case that the morphisms
$\ell$ and $e$ are isomorphisms.  
This definition is adapted from~\cite{Kock}, wherein a monad is called lax
monoidal if, in the above, such an $\ell$ and $e$ satisfy the first two
diagrams and the coherence conditions. 
(We caution that our notion differs from 
Kock's \emph{symmetric monoidal monad}, for which ``symmetric'' refers to
the monoidal tensor, rather than the monad.)

\begin{example} \label{Example:ProductOpLaxSymmetricComonad}
For every object $A$ in a cartesian monoidal category, the symmetric comonad
$(-)\times A$ is oplax monoidal, with structure
\[\begin{tikzcd}[ampersand replacement=\&]
  {X\times X'\times Y} \&\& {X\times X'\times Y\times Y} \& 
  {X\times Y\times X'\times Y} 
  \arrow["{\id\times \id\times \Delta_Y}", from=1-1, to=1-3]
	\arrow["\iso", from=1-3, to=1-4]
\end{tikzcd}\]
\[\begin{tikzcd}[ampersand replacement=\&]
  1 \times Y \& 1
	\arrow[from=1-1, to=1-2]
\end{tikzcd}\]
\end{example}

\subsection{Tensorial strengths}
Recall that monoidal monads have an induced (right)
\emph{tensorial strength}~(\cite{Kock}) defined as 
\[\begin{tikzcd}[ampersand replacement=\&]
	{\str_{A,B} = \ T(A) \otimes B} \& {T(A) \otimes T(B)} \& 
  {T(A\otimes B)}
	\arrow["{\id\otimes\eta_B}", from=1-1, to=1-2]
	\arrow["{\ell_{A,B}}", from=1-2, to=1-3]
\end{tikzcd}\]
that satisfies the following diagrams
\[\begin{tikzcd}[column sep = large, ampersand replacement=\&]
	{T^2(A)\otimes B} \& {T(T(A)\otimes B)} \& {T^2(A\otimes B)} \& 
  \\
	{T(A) \otimes B} \&\& {T(A\otimes B)} 
	\arrow["{\str_{T(A),B}}", from=1-1, to=1-2]
	\arrow["{T(\str_{A,B})}", from=1-2, to=1-3]
	\arrow["{\mu_A\otimes \id}"', from=1-1, to=2-1]
	\arrow["{\str_{A,B}}"', from=2-1, to=2-3]
	\arrow["{\mu_{A\otimes B}}", from=1-3, to=2-3]
\end{tikzcd}\]
\[\begin{tikzcd}[column sep = large, ampersand replacement=\&]
  {A\otimes B} \\
  {T(A) \otimes B} \& {T(A\otimes B)}
	\arrow["{\eta_A \otimes \id}"', from=1-1, to=2-1]
	\arrow["{\str_{A,B}}"', from=2-1, to=2-2]
	\arrow["{\eta_{A\otimes B}}", from=1-1, to=2-2]
\end{tikzcd}\]
and note that in the case that $T$ is a monoidal symmetric monad it also
satisfies
\[\begin{tikzcd}[column sep = large, ampersand replacement=\&]
	{T^2(A)\otimes B} \& {T(T(A)\otimes B)} \& {T^2(A\otimes B)} \\
	{T^2(A)\otimes B} \& {T(T(A)\otimes B)} \& {T^2(A\otimes B)}
	\arrow["{\str_{T(A),B}}", from=1-1, to=1-2]
	\arrow["{T(\str_{A,B})}", from=1-2, to=1-3]
	\arrow["{\varsigma_A\otimes \id}"', from=1-1, to=2-1]
	\arrow["{\varsigma_{A\otimes B}}", from=1-3, to=2-3]
	\arrow["{\str_{T(A),B}}"', from=2-1, to=2-2]
	\arrow["{T(\str_{A,B})}"', from=2-2, to=2-3]
\end{tikzcd}\]

Analogously, there is also a (left) tensorial strength defined as 
\[\begin{tikzcd}[ampersand replacement=\&]
	{\str'_{A,B} = \ A \otimes T(B)} \& {T(A) \otimes T(B)} \& 
  {T(A\otimes B)}
	\arrow["{\eta_A\otimes\id}", from=1-1, to=1-2]
	\arrow["{\ell_{A,B}}", from=1-2, to=1-3]
\end{tikzcd}\]

\begin{example} \label{Example:MainExample}
For every exponentiable object $A$ in a cartesian category, the symmetric
monad $(-)^A$, being a right adjoint, preserves finite products and, in
particular, it is a cartesian-monoidal symmetric monad.  We therefore have
cartesian tensorial strengths:
\[\begin{tikzcd}[ampersand replacement=\&, column sep=large]
  X^A\times Y \& (X\times Y)^A
  \\
  Y \times X^A \& (Y\times X)^A
	\arrow["{\str_{X,Y}}", from=1-1, to=1-2]
	\arrow["{\str'_{Y,X}}"', from=2-1, to=2-2]
  \arrow["{\iso}"', from=1-1, to=2-1]
	\arrow["{\iso}", from=1-2, to=2-2]
\end{tikzcd}\]
\end{example}

\section{Symmetric distributive laws}
\label{Section:SymmetricDistributiveLaws}

\begin{definition}
Let $\lscat{C}$ be a category, $(T, \mu, \eta, \varsigma)$ be a symmetric
monad on $\lscat{C}$, and $F$ be an endofunctor on $\lscat{C}$. A
\emph{symmetric distributive law} is a natural transformation 
$\psi: TF \to FT$ making the following diagrams commute:
\[
\begin{tikzcd}[ampersand replacement=\&]
  {T^2F} \& TFT \& {FT^2} \\
  TF \&\& FT
	\arrow["{T\psi}", from=1-1, to=1-2]
	\arrow["{\psi_T}", from=1-2, to=1-3]
	\arrow["{\mu_F}"', from=1-1, to=2-1]
	\arrow["\psi"', from=2-1, to=2-3]
	\arrow["{F\mu}", from=1-3, to=2-3]
\end{tikzcd}
\qquad\qquad
\begin{tikzcd}[ampersand replacement=\&]
	F \& TF \\
	\& FT 
	\arrow["{\eta_F}", from=1-1, to=1-2]
	\arrow["{F\eta}"', from=1-1, to=2-2]
	\arrow["\psi", from=1-2, to=2-2]
\end{tikzcd}
\]
\[\begin{tikzcd}[ampersand replacement=\&]
	{T^2F} \& TFT \& {FT^2} \\
	{T^2F} \& TFT \& {FT^2}
	\arrow["{T\psi}", from=1-1, to=1-2]
	\arrow["{\psi_T}", from=1-2, to=1-3]
	\arrow["{\varsigma_F}"', from=1-1, to=2-1]
	\arrow["{T\psi}"', from=2-1, to=2-2]
	\arrow["{\psi_T}"', from=2-2, to=2-3]
	\arrow["{F\varsigma}", from=1-3, to=2-3]
\end{tikzcd}\]
\end{definition}

The first two diagrams ask $\psi$ to be a distributive law between the
underlying monad $T$ and the endofunctor $F$, while the third asks $\psi$ to
respect $\varsigma$.  One may similarly define a 
\emph{symmetric codistributive law} between a symmetric comonad $T$ and an
endofunctor $F$.

\begin{examples} \label{Example:DistributiveLaw}
\begin{enumerate}
\item \label{Example:SelfDistributiveLaw}
For a symmetric monad $(T, \mu, \eta, \varsigma)$, the natural transformation
$\varsigma$ is a symmetric distributive law between $T$ and itself. 

\item \label{Example:CoDistributiveLaw}
For an endofunctor $F$ with cartesian strength 
$\str_{A, B}: F(A)\times B \to F(A \times B)$ each component $\str_{(-),B}$ is
a symmetric codistributive law between the symmetric comonad $(-)\times B$ and
the endofunctor $F$. 
\end{enumerate}
\end{examples}

\medskip
We highlight a simple lemma to be needed later that illustrates the tensoring
of two symmetric (co)distributive laws.

\begin{lemma} \label{Lemma:DistLawTensor}
Let $(T, \ell, e, \mu, \eta, \varsigma)$ be an oplax monoidal symmetric
(co)monad on a monoidal category $(\lscat{C}, \otimes, I)$, and let $G_1$ and
$G_2$ be two endofunctors on $\lscat{C}$. If $\psi_1: TG_1 \to G_1T$ and 
$\psi_2: TG_2 \to G_2T$ are symmetric (co)distributive laws, then
\[
	\psi_{1,2} \, =
  \begin{array}{c}
  \big(
  \begin{tikzcd}[ampersand replacement=\&]
	{T(G_1\otimes G_2)} 
  \& {TG_1\otimes TG_2} \& {G_1T\otimes G_2T = (G_1\otimes G_2)T}
	\arrow["{\ell_{G_1, G_2}}", from=1-1, to=1-2]
	\arrow["{\psi_1\otimes \psi_2}", from=1-2, to=1-3]
  \end{tikzcd}
  \big)
  \end{array}
\]
is a symmetric (co)distributive law between the symmetric (co)monad $T$ and
the endofunctor $G_1\otimes G_2$.  
\end{lemma}

\begin{remark} \label{Remark:Lemmas}
In the lemma above, note that if $\ell$ and $\psi_1$, $\psi_2$ are isomorphisms
then so is $\psi_{1,2}$.  
\end{remark}

\begin{definition}
For an endofunctor $F$ on a monoidal category, we let $F\pt$ be the
endofunctor given by 
\[
  F\pt(X) = F(X) \otimes X
\]
\end{definition}

\begin{examples} \label{Example:MainExampleContinued}
\begin{enumerate}
\item
For an endofunctor $F$ with cartesian strength 
$\str_{X,Y}: F(X)\times Y\to F(X\times Y)$, using the oplax cartesian-monoidal
symmetric comonad ${(-)\times Y}$
(Example~\ref{Example:ProductOpLaxSymmetricComonad}) 
and the symmetric codistributive law $\str_{(-),Y}$ between it and $F$
(Example~\ref{Example:DistributiveLaw}(\ref{Example:CoDistributiveLaw})), we
obtain, from Lemma~\ref{Lemma:DistLawTensor}, a symmetric codistributive law
\[
{\str\pt}_{(-),Y}: 
  (-\times Y)\icomp F\pt \to F\pt \icomp(- \times Y)
\]
It further follows that $\str\pt$ is a cartesian strength for $F\pt$;
explicitly, this is given by 
\[
  \!\!\!\!\!\!\!\!\!\!
  F(X)\times X \times Y
  \xrightarrow{\,\id\times\Delta\,} 
  F(X)\times X \times Y \times Y
  \iso 
  F(X)\times Y \times X \times Y
  \xrightarrow{\,\str\times\id\,} 
  F(X\times Y) \times X \times Y
\]

\item
For an oplax monoidal symmetric monad $T$, with symmetry $\varsigma$, from
Example~\ref{Example:DistributiveLaw}(\ref{Example:SelfDistributiveLaw})
and Lemma~\ref{Lemma:DistLawTensor}, we obtain a symmetric distributive law
\[
\underline\varsigma 
  : T \icomp T\pt \to T\pt \icomp T
\]
explicitly given by 
\[
  T\icomp T\pt(X)
  =
  T(T(X)\otimes X)
  \xrightarrow{\,\ell\,}
  TT(X)\otimes T(X)
  \xrightarrow{\,\varsigma\otimes\id\,}
  TT(X)\otimes T(X)
  =
  T\pt \icomp T(X)
\]

\end{enumerate}
\end{examples}

\section{Substitution algebras}
\label{Section:SubstitutionAlgebras}
  
As recalled in the introduction, in~\cite{FPT} the algebraic structure of
(single variable) substitution and (generic) variables on an object $A$ in the
object-classifier topos 
was axiomatized by means of two operations 
\begin{center}
  $A^\Var \times A \to A$ \enspace and \enspace $1 \to A^\Var$
\end{center}
subject to four equational laws (see also the last four equations
in~\ref{SAEquationalPresentation}).  
These capture the following properties:
$(i)$~the substitution on a variable performs the action;
$(ii)$~the substitution of a variable is a contraction;
$(iii)$~the substitution for an absent variable has no effect; 
and 
$(iv)$~the substitution operation is associative.
Next, we reconsider and generalize that definition streamlined in the
framework of the previous two sections.

\begin{definition} \label{Definition:SubstitutionAlgebra}
A \Def{$T$-substitution algebra} for a cartesian-monoidal symmetric monad 
$(T, \ell, e, \mu, \eta, \varsigma)$ on a cartesian category $\lscat C$, is an 
object $A \in \lscat C$ together with morphisms $s: T\pt(A) \to A$ and 
$v: 1 \to T(A)$ in $\lscat C$ such that the following diagrams commute:
\begin{equation}\label{LeftVarAxiom}
\begin{tikzcd}[ampersand replacement=\&]
	{1 \times A} \& 
  \\
  {T\pt(A)} \& A 
  \arrow["{v \times \id}"', from=1-1, to=2-1]
	\arrow["s"', from=2-1, to=2-2]
	\arrow["{\pi_2}", "\iso"', from=1-1, to=2-2]
\end{tikzcd}
\end{equation}
\begin{equation}\label{ContractAxiom}
\begin{tikzcd}[ampersand replacement=\&]
  {T^2(A) \times 1} \&\& {T^2(A)} 
  \\
  {T^2(A) \times T(A)} \& 
  {T\icomp T\pt(A)} \& {T(A)} 
	\arrow["{T(s)}"', from=2-2, to=2-3]
	\arrow["{\pi_1}", "\iso"', from=1-1, to=1-3]
	\arrow["{\mu}", from=1-3, to=2-3]
	\arrow["{\id \times v}"', from=1-1, to=2-1]
	\arrow["{\ell}"', "\iso", from=2-1, to=2-2]
\end{tikzcd}
\end{equation}
\begin{equation*}\label{WeakeningAxiom}
  \begin{tikzcd}[ampersand replacement=\&]
  {A \times A} \& 
  \\
  {T\pt(A)} \& A 
  \arrow["{\pi_1}", from=1-1, to=2-2]
	\arrow["{\eta\times \id}"', from=1-1, to=2-1]
	\arrow["s"', from=2-1, to=2-2]
\end{tikzcd}
\end{equation*}
\[\begin{tikzcd}[column sep = huge, ampersand replacement=\&]
	{T\icomp T\pt(A)\times A} \&
  {T\pt\icomp T(A)\times A} \& 
  {T\pt\icomp T\pt(A)} \& {T\pt(A)} 
  \\
	{T\pt(A)} \&\&\& A
	\arrow["{\underline\varsigma\times\id}", "\iso"', from=1-1, to=1-2]
	\arrow["s"', from=2-1, to=2-4]
	\arrow["{\str\pt}", from=1-2, to=1-3]
	\arrow["{T(s)\times\id}"', from=1-1, to=2-1]
	\arrow["{T\pt(s)}", from=1-3, to=1-4]
	\arrow["s", from=1-4, to=2-4]
\end{tikzcd}\]
\end{definition}

Diagram~(\ref{LeftVarAxiom}) may be naturally considered a left-unit law.  As
we now show, diagram~(\ref{ContractAxiom}) is equivalent to a right-unit law.

\begin{proposition}
Diagram~(\ref{ContractAxiom}) commutes if, and only if, so does the following
one:
\begin{equation}\label{EvalAxiom}
\begin{tikzcd}[ampersand replacement=\&]
  \& T(A)\times 1 \& 
  \\
  T(A)\times T(A) \& T\icomp T\pt(A) \& T(A)
	\arrow["\id\times v"', from=1-2, to=2-1]
	\arrow["\str'"', from=2-1, to=2-2]
	\arrow["T(s)"', from=2-2, to=2-3]
	\arrow["\pi_1", "\iso"', from=1-2, to=2-3]
\end{tikzcd}
\end{equation}
\end{proposition}
\begin{proof}
($\Rightarrow$)
Diagram~(\ref{EvalAxiom}) is readily obtained from
diagram~(\ref{ContractAxiom}) by precomposition with the morphism 
$\eta_{T}\times\id:T(A)\times 1\to T^2(A)\times 1$.

($\Leftarrow$)
We provide a diagrammatic proof.
\[\begin{tikzcd}[ampersand replacement=\&,cramped,column sep=scriptsize, row
    sep=large] 
  \& {T^2(A) \times 1} 
	\\
	{T^2(A) \times T(A)} \& {T^2(A) \times T(1)} \& {T(T(A) \times 1)} 
  \\
	\& {T^2(A) \times T^2(A)} \& {T(T(A) \times T(A))} \& {T^2T^\bullet(A)} 
  \& {T^2(A)} 
  \\
	\& {T^3(A) \times T^2(A)} \& {T(T^2(A)\times T(A))} 
  \\
	{T^2(A) \times T(A)} \&\&\& {TT^\bullet(A)} \& {T(A)}
	\arrow["{\mathrm{id} \times v}"', from=1-2, to=2-1]
	\arrow["{\mathrm{id} \times \eta_1}"{pos=0.7}, from=1-2, to=2-2]
	\arrow["\cong"'{pos=0.65}, draw=none, from=1-2, to=2-2]
	\arrow["{\mathrm{str}}"{pos=0.4}, from=1-2, to=2-3]
	\arrow["\cong"'{pos=0.5}, draw=none, from=1-2, to=2-3]
	\arrow["{\pi_1}", curve={height=-50pt}, from=1-2, to=3-5]
	\arrow["\cong"', curve={height=-50pt}, draw=none, from=1-2, to=3-5]
	\arrow["{\mathrm{id}\times \eta_T}"', from=2-1, to=3-2]
	\arrow["{\mathrm{id} \times \mathrm{id}}"', from=2-1, to=5-1]
	\arrow["\ell", from=2-2, to=2-3]
	\arrow["\cong"', draw=none, from=2-2, to=2-3]
	\arrow["{\mathrm{id} \times T(v)}"', from=2-2, to=3-2]
	\arrow["{T(\mathrm{id} \times v)}"', from=2-3, to=3-3]
	\arrow["{T(\pi_1)}"{pos=0.25}, from=2-3, to=3-5]
	\arrow["\cong"'{pos=0.3}, draw=none, from=2-3, to=3-5]
	\arrow["\ell", from=3-2, to=3-3]
	\arrow["\cong"', draw=none, from=3-2, to=3-3]
	\arrow["{T(\eta_T) \times \mathrm{id}}"', from=3-2, to=4-2]
	\arrow["{T(\mathrm{str}')}", from=3-3, to=3-4]
        \arrow["{T(\eta_T \times \id)}"', from=3-3, to=4-3]
	\arrow["{T^2(s)}"', from=3-4, to=3-5]
	\arrow["{\mu_{T^\bullet}}"', from=3-4, to=5-4]
	\arrow["\mu", from=3-5, to=5-5]
	\arrow["\ell"', from=4-2, to=4-3]
	\arrow["\cong", draw=none, from=4-2, to=4-3]
	\arrow["{\mu_T\times \mu}"', from=4-2, to=5-1]
	\arrow["{T(\ell)}"', from=4-3, to=3-4]
	\arrow["\cong", draw=none, from=4-3, to=3-4]
	\arrow["\ell"', from=5-1, to=5-4]
	\arrow["\cong", draw=none, from=5-1, to=5-4]
	\arrow["{T(s)}"', from=5-4, to=5-5]
\end{tikzcd}\]
\end{proof}

\begin{theorem}\label{Theorem:AxiomatizationEquivalence}
$T$-substitution algebras may be equivalently axiomatized by replacing 
diagram~(\ref{ContractAxiom}) with diagram~(\ref{EvalAxiom}).
\end{theorem}

\begin{definition}
A homomorphism between $T$-substitution algebras $(A,s,v)$ and $(A',s',v')$ is
a morphism $h: A \to A'$ such that the following diagrams commute:
\begin{center}
$\begin{tikzcd}[ampersand replacement=\&]
	1 \& {T(A)} 
  \\
	\& {T(A')} 
	\arrow["v", from=1-1, to=1-2]
	\arrow["{v'}"', from=1-1, to=2-2]
	\arrow["{T(h)}", from=1-2, to=2-2]
\end{tikzcd}$
\qquad\qquad\qquad
$\begin{tikzcd}[ampersand replacement=\&]
  {T\pt(A)} \& A \\
  {T\pt(A')} \& {A'}
	\arrow["s", from=1-1, to=1-2]
	\arrow["{T\pt(h)}"', from=1-1, to=2-1]
	\arrow["{s'}"', from=2-1, to=2-2]
	\arrow["h", from=1-2, to=2-2]
\end{tikzcd}$
\end{center}
\end{definition}

\medskip
We are interested here in the specific category of substitution algebras
defined below.  In the following, recall Example~\ref{Example:MainExample}.

\begin{definition}
We let $\SubstAlg$ be the category of substitution algebras and homomophisms
between them for the cartesian-monoidal symmetric monad $(-)^V$ on the
object-classifier topos~$\FF=\Set^{\F}$, for $V = \F(1,-)$ the universal
object model.  
\end{definition}

\begin{remark}
The equivalence of Theorem~\ref{Theorem:AxiomatizationEquivalence} in the
context of $\SubstAlg$ accounts for the respective axiomatizations considered
in~\cite{FPT} and in~\cite{FioreStaton}.
\end{remark}

\subsection{} 
Since the universal object model $V\in\FF$ is the representable at $1\in\F$,
which with the structure $(\cont:2\to1,\weak:0\to1,\symm:2\to2)$ in $\F$ is
the universal symmetric monoid 
(recall~\ref{Subsection:MainSymmetricMonoid}), 
the exponentiation by $V$ symmetric monad is equivalently described by
$(\delta,\moncont,\monweak,\monswap)$ given as 
follows \big(see
Example~\ref{Example:SymmetricMonoidMonad}(\ref{Example:ConcreteSymmetricMonoidMonad})\big):
\begin{align*}
\delta(A) =\ 
& A(-+1) 
\\[1mm]
\moncont_A = \ & A(- + \cont) : A(-+2) \to A(-+1)
\\[1mm]
\monweak_A = \ & A(- + \weak) : A(-) \to  A(-+1)
\\[1mm]
\monswap_A = \ & A(- + \swap) : A(-+2)\to A(-+2)
\end{align*}
We 
henceforth adopt this representation.  In particular, substitution-algebra
structure $s:A^V\times A\to A$ and $v: 1\to A^V$ on $A\in\FF$ is therefore
given by natural transformations
\[
  s_m : A(m+1)\times A(m) \to A(m)
  \enspace,\quad
  v_m : 1 \to A(m+1)
  \qquad\qquad
  (m\in\F)
\]

\subsection{Equational presentation} \label{SAEquationalPresentation}
It follows from the above that the category of substitution algebras 
$\SubstAlg$ in the object-classifier topos $\FF$ has a countably-sorted
equational presentation.  
Indeed, this has set of sorts $\Nat$ and operators 
\[\begin{array}{lll}
  \alpha_f : & m \to n & \quad\big(m,n\in\Nat, f\in\F(m,n)\big) 
  \\[1mm]
  \varsigma_m : & m+1,m \to m & \quad(m\in\Nat) 
  \\[1mm]
  \nu : & 1 & 
\end{array}\]
subject, for all $\ell,m,n\in\Nat$, $g\in\F(\ell,m)$, $f\in\F(m,n)$, to the
equations:
\[\begin{array}{l}
  \begin{array}{l}
  x : \ell \ \vdash \alpha_f\big(\alpha_g(x)\big) = \alpha_{f g}(x)
  : n 
  \end{array}
  \\[2.5mm]
  \begin{array}{l}
  x : m \ \vdash \alpha_{\id_m}(x) = x : m 
  \end{array}
  \\[2.5mm]
  \begin{array}{l}
  x : m+1, y : m \
  \vdash
  \alpha_f\big(\varsigma_m(x,y)\big) 
  = 
  \varsigma_n\big(\alpha_{f+\id_1}(x),\alpha_f(y)\big) 
  : n 
  \end{array}
  \\[2.5mm]
  \begin{array}{l}
  x : m \ \vdash \varsigma_m\big(\nu_m,x\big) = x : m
  \end{array}
  \\[2.5mm]
  \begin{array}{l}
  x : m+2 \ 
  \vdash 
  \varsigma_{m+1}\big(x,\nu_m\big) = \alpha_{\id_m+\cont}(x)
  : m+1
  \end{array}
  \\[2.5mm]
  \begin{array}{l}
  x : m , y : m \ \vdash \varsigma_m\big(\alpha_{\id_m+\weak}(x),y\big) = x : m
  \end{array}
  \\[2.5mm]
  \begin{array}{l}
  x : m+2, y : m+1 , z : m \
  \\[1mm]
  \quad\vdash \ 
  \varsigma_m\big(\varsigma_{m+1}(x,y),z\big)
  =
  \varsigma_m
  \big( \varsigma_{m+1}
          \big( \alpha_{\id_m+\swap}(x) 
                ,
                \alpha_{\id_m+\weak}(z) \big)
      , 
      \varsigma_m(y,z) \big)
  : m
  \end{array}
\end{array}\]
where $\nu_m =  \alpha_{(0\,\mapsto m)}(\nu) : m+1$.

The operators $\alpha_f$ together with the first two equations correspond
to the presheaf structure of objects in $\FF$, the next equation corresponds
to the naturality of the substitution operation as embodied by the operators
$\varsigma_m$, and the last four equations correspond to the laws of
substitution algebras.

\begin{corollary}
The category of substitution algebras is countably-sorted algebraic.
\end{corollary}

\section{Isomorphism theorem}
\label{Section:IsomorphismTheorem}

We show that the categories of substitution algebras $\SubstAlg$ and of
abstract clones $\AbsClon$ are isomorphic.  Since, as recalled in the
introduction, the category of Lawvere theories is equivalent to that of
abstract clones, 
this establishes the main aim of the paper in exhibiting a finite equational
presentation of Lawvere theories in the object-classifier topos.

We will construct inverse functors $\ACtoSA : \AbsClon \to \SubstAlg$ and 
$\SAtoAC: \SubstAlg \to \AbsClon$. 
The idea of these constructions is simple:  For abstract clones and
substitution algebras, respectively understood as modelling simultaneous and
single-variable substitution, the first functor expresses single-variable
substitution as a special case of simultaneous substitution while the second
functor expresses simultaneous substitution as iterated application of
single-variable substitution.  

\subsection{From abstract clones to substitution algebras}\phantom{}

\medskip\noindent
Let $(C, \mu, \iota)$ be an abstract clone. 

\subsubsection{}
For $f\in\F(m,n)$, define
\[
  C(f) : C_m \to C_n 
  : t \mapsto \mu_{m,n}\big(t,\iota^n_{f(0)},\ldots,\iota^n_{f(m-1)}\big)
\]
Then, the families of sets $\setof{\, C_m \,}_{m\in\Nat}$ and of functions
$\setof{\, C(f) : C_m \to C_n \,}_{m,n\in\Nat, f\in\F(m,n)}$ determine a
presheaf $C$ in $\FF$.

\subsubsection{}
For $m\in\Nat$, define
\[
  v_m = \iota^{m+1}_{m}\in C(m+1)
\]
Then, the family of elements $\setof{\, v_m \in C(m+1) \,}_{m\in\Nat}$
determines a natural transformation $v: 1\to\delta(C)$ in $\FF$.

\subsubsection{}
For $m\in\Nat$, define
\[
  s_m : C(m+1) \times C(m) \to C(m) 
  : (t,u) 
    \mapsto 
    \mu_{m+1,m}\big( t, \iota^m_{0}, \ldots, \iota^m_{m-1}, u \big)
\]
Then, the family of functions 
$\setof{\, s_{m,n} : C(m+1) \times C(m) \to C(m) \,}_{m\in\Nat}$ determines a
natural transformation $s: \delta(C)\times C\to C$ in $\FF$.

\subsubsection{}
The structure $\ACtoSA(C,\mu,\iota)=(C,s,v)$ is a substitution algebra and for
an abstract-clone homomorphism $h:(C,\mu,\iota)\to(C',\mu',\iota')$, the
family of functions $\ACtoSA(h)=\setof{\, h_m : C(m) \to C'(m) \,}_{m\in\Nat}$
is a substitution-algebra homomorphism 
$\ACtoSA(C,\mu,\iota)\to \ACtoSA(C',\mu',\iota')$.  This construction defines
a functor $\ACtoSA:\AbsClon\to\SubstAlg$.

\subsection{From substitution algebras to abstract clones}\phantom{}

\medskip\noindent
Let $(A,s,v)$ be a substitution algebra.

\subsubsection{}
For $m,n\in\Nat$, define
\[
  \varphi_{m,n} : A(n+m)\times (A\,n)^m \to A(n)
\]
by induction as 
\[
  \varphi_{0,n} 
  = 
  \begin{tikzcd}
  A(n)\times(A\,n)^0 \ar[r,"\pi_1","\iso"'] & A(n)
  \end{tikzcd}
\]
and
\[\begin{tikzcd}[column sep = huge, row sep = 17.5pt, ampersand replacement=\&]
  A(n+m+1)\times(A\,n)^m\times A(n) \& A(n)
  \\
  A(n+m+1)\times A(n)\times(A\,n)^m \& 
  \\
  A(n+m+1)\times A(n+m)\times(A\,n)^m \& A(n+m)\times(A\,n)^m
	\arrow["\iso"', from=1-1, to=2-1]
	\arrow["\varphi_{m+1,n}", from=1-1, to=1-2]
	\arrow["\id\times A(i\,\mapsto i)\times\id"', from=2-1, to=3-1]
	\arrow["\varphi_{m,n}"', from=3-2, to=1-2]
	\arrow["s_{n+m}\times\id"', from=3-1, to=3-2]
\end{tikzcd}\]

\subsubsection{}
For $m,n\in\Nat$, define
\[
  \mu_{m,n} : A(m)\times (A\,n)^m \to A(n)
\]
as the composite
\[\begin{tikzcd}[column sep = large, ampersand replacement=\&]
  {A(m)\times (A\,n)^m} \&\& {A(n+m)\times (A\,n)^m} \& {A(n)}
  \arrow["A(i\,\mapsto n+i)\times\id", from=1-1, to=1-3]
  \arrow["\varphi_{m,n}", from=1-3, to=1-4]
\end{tikzcd}\]

\subsubsection{}
For $m\in\Nat$ and $i\in\ord m$, define 
\[
  \iota^m_i 
  = A\big(0\mapsto i\big)\big(\,v_0(\,)\,\big) 
  \in A(m)
\]

\subsubsection{}
The structure $\SAtoAC(A,s,v)=(A,\mu,\iota)$ is an abstract clone and for a
substitution-algebra homomorphism $h:(A,s,v)\to(A',s',v')$, the family of
functions $\SAtoAC(h)=\setof{\, h_m : A(m) \to A'(m)\,}_{m\in\Nat}$ is an
abstract-clone homomorphism $\SAtoAC(A,s,v)\to\SAtoAC(A',s',v')$.  This
construction defines a functor $\SAtoAC: \SubstAlg\to\AbsClon$.

\begin{theorem}
The functors $\ACtoSA: \AbsClon\to\SubstAlg$ and 
$\SAtoAC: \SubstAlg\to\AbsClon$ form an isomorphism of categories.  
\end{theorem}

\begin{corollary}
Substitution algebras provide a finite algebraic presentation of Lawvere
theories in the object-classifier topos.
\end{corollary}

\medskip
\subsection*{Acknowledgement}
We are grateful to the referee for comments that led to improving the paper.

%
%
%

\refs

\bibitem [{Ad\'amek, Rosick\'y, and Vitale}, 2011]{AdamekRosickyVitale} 
J Ad\'amek, J Rosick\'y, and E M Vitale,
Algebraic Theories: A Categorical Introduction to General Algebra.
Cambridge University Press, 2011.

\bibitem [Cohn, 1981]{Cohn} 
P M Cohn, 
Universal Algebra.  
Mathematics and its Applications Vol~6 (2nd Ed).  
D~Reidel Publishing Co, 1981.

\bibitem [Day, 1970]{Day}
B Day,
On closed categories of functors.
\emph{Reports of the {M}idwest {C}ategory {S}eminar, {IV}}, 
Lecture Notes in Mathemathics, Vol~137, pp~1--38.  
Springer, 1970. 

\bibitem [Fiore, 2005]{FOSSACS05}
M Fiore, 
Mathematical models of computational and combinatorial structures.
Invited talk at the {\em Foundations of Software Science and Computation
Structures (FOSSACS 2005) Conference}, 
Lecture Notes in Computer Science, Vol~3441, pp~25--46.  
Springer, 2005.  

\bibitem [Fiore, 2006]{MFPS06}
M Fiore, On the structure of substitution. 
Slides for an invited talk at the {\em 22nd Annual Conference on Mathematical
Foundations of Programming Semantics (MFPS XXII)}, 2006.

\bibitem [{Fiore, Plotkin, and Turi}, 1999]{FPT}
M Fiore, G Plotkin, and D Turi, 
Abstract syntax and variable binding.  
In {\em 14th {S}ymposium on {L}ogic in {C}omputer {S}cience}, pp~193--202.  
IEEE Computer Society, 1999.

\bibitem [{Fiore and Staton}, 2014]{FioreStaton}
M Fiore and S Staton,
Substitution, jumps, and algebraic effects.
In {\em Proceedings of the Joint Meeting of the Twenty-Third EACSL Annual
Conference on Computer Science Logic (CSL) and the Twenty-Ninth Annual
ACM/IEEE Symposium on Logic in Computer Science (LICS)}, ACM 2014.

\bibitem [Grandis, 2001]{Grandis}
M Grandis, 
Finite sets and symmetric simplicial sets.  
\emph{Theory and Applications of Categories}, Vol~8, No~8, pp~244--252, 2001.

\bibitem [{Gr\"atzer}, 2008]{Gratzer} 
G Gr\"atzer, 
Universal Algebra (2nd Ed). 
Springer, 2008.

\bibitem [Im and Kelly, 1986]{ImKelly}
G B Im and G M Kelly, 
A universal property of the convolution monoidal structure.
\emph{J~Pure Appl~Algebra}, Vol~43, No~1, pp~75--88, 1986. 
 
\bibitem [Kock, 1972]{Kock}
A Kock, 
Strong functors and monoidal monads. 
\emph{Arch Math} 23, pp~113-–120, 1972. 

\bibitem [Lawvere, 1963]{LawverePhDThesis} 
F W Lawvere, 
Functorial semantics of algebraic theories.
PhD thesis, Columbia University, 1963.
(Republished in: 
\emph{Reprints in Theory and Applications of Categories}, No~5, 2004.)

\bibitem [Lawvere, 1969]{LawvereOrdSum} 
F W Lawvere, 
Ordinal sums and equational doctrines.
In \emph{Seminar on Triples and Categorical Homology Theory}, 
Lecture Notes in Mathematics, Vol~80.  
Springer, 1969.

\bibitem [Taylor, 1973]{Taylor73} 
W Taylor, 
Characterizing Mal’cev conditions. 
\emph{Algebra Univ}, Vol~3, pp~51--397, 1973. 

\bibitem [Taylor, 1993]{Taylor93} 
W Taylor, 
Abstract Clone Theory. 
In {\em Algebras and Orders}, pp~507--530. 
Springer, 1993.

\bibitem [Wraith, 1969]{Wraith} 
G Wraith, Algebraic theories.  
Aarhus Lecture Notes Series, No~22, 1969. 

\endrefs

\end{document}